\documentclass{amsart}

\usepackage{graphicx}
\usepackage{amssymb}
\usepackage{amsmath}
\usepackage{amsthm}
\usepackage{enumerate}
\usepackage[dvipsnames]{xcolor}

\newtheorem{theorem}{Theorem}[section]
\newtheorem{corollary}[theorem]{Corollary}
\newtheorem{lemma}[theorem]{Lemma}

\theoremstyle{definition}
\newtheorem{definition}[theorem]{Definition}
\newtheorem{definitions}[theorem]{Definitions}

\newtheorem{example}[theorem]{Example}
\newcommand{\ep}{\varepsilon}

\begin{document}

\title{Shadowable Points for flows}

\author{J. Aponte}
\address{Departamento de Matem\'{a}tica, Universidade Federal do Rio de Janeiro, Rio de Janeiro, Brazil.}
\email{aponteg.jesus@ufrj.br}
\thanks{Partially supported by CAPES from Brazil.}

\author{H. Villavicencio}
\address{Instituto de Matem\'{a}tica y Ciencias Afines, Lima, Per\'{u}.}
\email{hvillavicencio@imca.edu.pe}
\thanks{Partially supported by FONDECYT from Peru (C.G. 217--2014).}

%    General info


\begin{abstract}
A shadowable point for a flow is a point where the shadowing lemma holds for
pseudo-orbits passing through it. We prove that this concept satisfies the following properties: the set of shadowable points is invariant and a $G_{\delta}$ set. A flow has the pseudo-orbit tracing property if and only if every point is shadowable. The chain recurrent and nonwandering sets coincide when every chain recurrent point is shadowable. The chain recurrent points which are shadowable are exactly those that can be are approximated by periodic points when the flow is expansive. We study the relations between shadowable points of a homeomorphism and the shadowable points of its suspension flow. We characterize the set of forward shadowable points for transitive flows and chain transitive flows. We prove that the geometric Lorenz attractor does not have shadowable points. We show that in the presence of shadowable points chain transitive flows are transitive and that transitivity is a necessary condition for chain recurrent flows with shadowable points whenever the phase space is connected. Finally, as an application these results we give concise proofs of some well known theorems establishing that flows with POTP admitting some kind of recurrence are minimal. These results extends those presented in \cite{Morales16}.
\end{abstract}
\maketitle
\section{Introduction}

\noindent
The theory of shadowing plays an important role in the qualitative theory of dynamical systems.
%The theory of shadowing in dynamical systems 
It has been largely studied by many researchers and is well documented (see for instance \cite{Palmer2000}). It refers to the general problem of approximating orbits obtained in the presence of noise or round-off error (for instance solutions obtained by numerical computations). There are several ways to define the {\em shadowing property} for flows, see for instance \cite{Pilyugin2008} and references therein. In essence, the central idea among the majority of definitions of shadowing for flows is the following: even if small errors occur at each iteration, one can track the resulting pseudo-orbit by a true orbit with a time reparametrization.

Recently, in \cite{Morales16} the definition of shadowing for homeomorphisms in a compact metric space was generalized by splitting the shadowing property into pointwise shadowings  giving rise to the concept of shadowable points, which are points where the shadowing property holds for pseudo-orbits passing through them. In \cite{Kawaguchi17} the author further extends this notion by introducing the concept of quantitative shadowable points for homeomorphism and some important question made in \cite{Morales16} were answered. In light of these results, it is natural to consider a notion of shadowable points for flows and expect similar results to the homeomorphism case. 

Hence we introduce the concept of shadowable points for flows and we prove that this notion satisfies the following properties: the set of shadowable points is invariant and a $G_{\delta}$ set. A flow has the pseudo-orbit tracing property if and only if every point is shadowable. The chain recurrent and nonwandering sets coincide when every chain recurrent point is shadowable. The chain recurrent points which are shadowable are exactly those that can be are approximated by periodic points when the flow is expansive.  We study the relations between the set of shadowable points of a homeomorphism and the sets of shadowable points of its suspension flow. We characterize the set of forward  shadowable points for transitive flows and chain transitive flows and we prove that the geometric Lorenz attractor does not have shadowable points. We show that in the presence of shadowable points chain transitive flows are transitive and that transitivity is a necessary condition for chain recurrent flows with shadowable points whenever the phase space is connected. Finally, as an application these results we give a concise proof of some well known theorems establishing that flows with POTP admitting some kind of recurrence are minimal.

\section{Statement of results}

Hereafter $(X,d)$ will denote a compact metric space. The closure operation will be denoted by $\overline{(\cdot)}$. A {\em flow} of $X$ is a map $\phi\colon X\times \mathbb{R}\rightarrow X$ satisfying $\phi(x,0)=x$ and $\phi(\phi(x, s),t)=\phi(x, s+t)$ for all $t, s\in \mathbb{R}$ and $x\in X$.
A flow is continuous if it is continuous with respect to the product metric of $X\times \mathbb{R}$. The \emph{time $t$-map} $\phi_t\colon X\rightarrow X$ defined by $\phi_t(x)=\phi(x,t)$ is a
homeomorphism of $X$ for all $t\in \mathbb{R}$. So, the flow $\phi$ can be interpreted as a family of homeomorphisms $\Phi=( \phi_t)_{t\in \mathbb{R}}$ such that
$\phi_0=id$ and $\phi_t\circ \phi_s=\phi_{t+s}$ for all $t,s\in \mathbb{R}$.  Given $A\subset X$ and $I\subset \mathbb{R}$ we set $\phi_I(A)=\{\phi_t(x):(t,x)\in I\times A\}$. If $A$ consists of a single point $x$, then we write
$\phi_I(x)$ instead of $\phi_I(\{x\})$. In particular, $\phi_{\mathbb{R}}(x)$ is called the orbit of $x\in X$ under $\phi$. By a \emph{periodic point} we mean a point $x\in X$ for which there is a minimal $t>0$ satisfying $\phi_{t}(x)=x$. This minimal $t$ is the so-called \emph{period} of $x$ and is denoted by $t_x$. Denote by $Per(\phi)$ the set of periodic points of $\phi$.

Given $\delta$, $T > 0$, $a\in \mathbb{Z}\cup\{-\infty\}$, $b\in\mathbb{Z}\cup\{+\infty\}$ with $a\leq b$, we say that a sequence of pairs $(x_i,\, t_i)_{i=a}^b$ in $X\times\mathbb{R}$ is a $(\delta,\, T)$-pseudo-orbit of $\phi$ if for all integer indexes $i$ such that $a\leq i \leq b-1$ we have that $t_i \geq T$ and
$d(\phi_{t_i}(x_i),\, x_{i+1}) \leq \delta$. If $a,b\in\mathbb{Z}$ and $ab\leq 0$, we say that it is a finite $(\delta,\, T)$-pseudo-orbit. If $a = 0$ and $b=\infty$ we say that it is a forward $(\delta,\,T)$-pseudo-orbit and if $a = 0$ and $b<\infty$ we say that it is a $(\delta,\,T)$-chain. (see \cite{Komuro84, Thomas84}). 

For any sequence of real numbers $(t_j)_{\,j\in \mathbb{Z}}$ we write
    \[s_i =
      \begin{cases}
        \displaystyle{\sum_{j=0}^{i-1}} t_j & i > 0,\\
		    0				             & i = 0, \\
		    -\displaystyle{\sum_{j=i}^{-1}} t_j & i < 0.
      \end{cases}
    \]

Let $(x_i, t_i)_{\,i\in \mathbb{Z}}$ be a $(\delta,\, T)$-pseudo-orbit of $\phi$ and let $t\in \mathbb{R}$, we denote by $x_0\star t$ a point in the $(\delta,\, T)$-pseudo-orbit $t$ {\em units from} $x_0$ \cite{Komuro84}. More precisely,
\[x_0\star t = \phi_{\,t\,-\,s_i}(x_i)\ \mbox{whenever } s_i\leq t < s_{i+1}.\]

Denote by $\mathrm{Rep}$ the set of surjertive strictly increasing maps $ h:\mathbb{R}\rightarrow \mathbb{R}$ such that $h(0)=0$ which will be called the set of reparameterizations.

Next, we recall the definition of pseudo orbit tracing property for flows \cite{Thomas84}.

\begin{definition}\label{defi01}
A flow $\phi$ on X is said to have the pseudo orbit tracing property, {\em POTP}, if for all $\ep>0$ there exists $\delta>0$ such that every  $(\delta,\,1)$-pseudo-orbit is $\ep$-shadowed by an orbit of $\phi$.
\end{definition}

Additionally, we recall the definition of {\em shadowable points} for homeomorphisms \cite{Morales16}. Let $f\colon X\to X$ a homeomorphism. Given $\delta>0$, we say that a bi-infinite
sequence $(x_n)_{n\in\mathbb{Z}}$ is a $\delta$-pseudo-orbit passing through the point $p\in X$, if $x_0=p$ and for every integer $n$ we have that $d(f(x_n),\,x_{n+1})\leq\delta$. We say $p$ is \emph{shadowable} if for each $\ep>0$ there exists a $\delta>0$  such that for every $\delta$-pseudo-orbit $(x_n)_{n\in\mathbb{Z}}$ passing through $p$, there is a point $q\in X$ such that $d(f^n(q),\,x_n)\leq\ep$ for all $n\in\mathbb{Z}$. The set of shadowable points of $f$ is denoted by $Sh(f)$.

Motivated by this we consider the notion of shadowing for pseudo orbits in flows passing through a given point.

\begin{definitions}\label{defi1}
Given positive numbers $\delta,\, T$ and $\ep$, we say that a $(\delta,\, T)$-pseudo-orbit $(x_i,\, t_i)_{i\in \mathbb{Z}}$ of $\phi$ passes through $p$ if $x_0=p$, and we say that is {\em $\ep$-shadowed}  if there are a point $y\in X$ and a function $h\in \mathrm{Rep}$ such that
$$d(x_0\star t,\, \phi_{h(t)}(y)) \leq \ep, \ \mbox{for each}\ t\in \mathbb{R}.$$
\end{definitions}

Now we introduce the main objects of study.

\begin{definition}\label{defi3}
A point $p\in X$ is \emph{shadowable with respect to the parameter $T>0$}, if for every $\ep > 0$, there exists $\delta > 0$ such that every $(\delta,\, T)$-pseudo-orbit of $\phi$ passing through $p$ can be $\ep$-shadowed. When $p$ is shadowable with respect to the parameter $T=1$ we say that $p$ is \emph{shadowable}.
\end{definition}
  
We denote by $Sh(\phi)$ the set of shadowable  points of $\phi$ in $X$. In what follows we will give some examples of shadowable points.

\begin{example}
If a flow $\phi$ on $X$ has the {\em POTP} then $Sh(\phi)= X$. The converse is also true as we will see shortly.
\end{example}

Our first result deals with the basic properties of shadowable points related to the following standard definitions. We say that a point $p\in X$ is \emph{non-wandering} if for every neighborhood $U$ of $p$ and every $T\in\mathbb{R}$ there is $t\geq T$ such that $\phi_{\,t}(U)\cap U\neq\emptyset$. Two points $p$ and $q$ are $(\delta,\,T)$-related if there are two  $(\delta,\,T)$-chains $(x_i,\,t_i)_{i=0}^m$ and  $(y_i,\,s_i)_{i=0}^n$ such that $p=x_0=y_n$ and $y=y_0=x_m$. We say that $p$ and $q$ are related (written $p \sim q$) if they are $(\delta,\,T)$-related for every $\delta,\,T >0$. A point $p$ is \emph{chain recurrent} if $p\sim p$. Denote by $\Omega(\phi)$ and $CR(\phi)$ the set of non-wandering and chain recurrent points of $\phi$ respectively. Clearly $\Omega(\phi)\subseteq CR(\phi)$ and the inclusion may be proper \cite{Alongi07}.
 
We say that a flow $\phi$ is {\em expansive} if for every $\varepsilon>0$ there exists $\delta>0$ with the property that if $d(\phi_t(x),\, \phi_{h(t)}(y))\leq \delta$ for all $t\in \mathbb{R}$, for every pair of points $x,y\in X$ and some $h\in \mathrm{Rep}$, then $y\in\phi_{(-\varepsilon,\,\varepsilon)}(x)$.  
 
A subset $\Lambda$ of $X$ is invariant under $\phi$ (or $\phi$-invariant) if  $\phi_t(\Lambda)= \Lambda$ for every $t\in \mathbb{R}$.  An \emph{equivalence} between continuous flows \cite{Thomas84}, $\phi$ on $X$ and $\psi$ on another metric space $Y$, is an homeomorphism $f:X\rightarrow Y$ carrying orbits of $\phi$ onto orbits of $\psi$ such that for every $x\in X$ there exists $h_x\in \mathrm{Rep}$ depending continuously on $x$ and satisfying 
$$f^{-1}(\psi(f(x),\,t))=\phi(x,\,h_x(t)) \ \mbox{for every}\  t\in \mathbb{R}.$$
In this case we say that the flows are \emph{equivalent}.

With these definitions we can state our first result.

\begin{theorem}\label{teo4}
Given a flow $\phi$ in a compact metric space $(X,\, d)$, the set of shadowable points satisfies the following properties:
\begin{enumerate}[(a)]
\item $Sh(\phi)$ is invariant possibly empty and noncompact.
\item The flow $\phi$  has the  POTP if and only if $Sh(\phi) = X$.
\item If $CR(\phi)\subseteq Sh(\phi)$ then $CR(\phi) = \Omega(\phi)$.
\item If $\phi$ is expansive and $CR(\phi)\subseteq Sh(\phi)$, then $CR(\phi)=\overline{Per(\phi)}$.
\item If $f$ is an equivalence between $\phi$ and $\psi$, then $f(Sh(\phi)) = Sh(\psi)$.
\end{enumerate}
\end{theorem}
Next, we give an example related to Theorem \ref{teo4}. We recall that a flow $\phi$ on $X$ is isometric if $d(\phi_t(x),\,\phi_t(y)) = d(x,\,y)$ for every $x,\,y\in X$ and each $t\in \mathbb{R}$, and  is minimal is all of its orbits are dense in $X$.

\begin{example}
If $\phi$ and $\psi$ are continuous flows on $X$, then it is not always true that $Sh(\phi)\times Sh(\psi)\subset Sh(\phi\times \psi)$. Indeed, if we consider  $\phi(z,\,t) = e^{2t\pi i}z$, defined in the unit circle $S^1$, this flow has the {\em POTP}. Then $Sh(\phi)=S^1$ by item $(b)$ of Theorem \ref{teo4}. If the inclusion holds then $Sh(\phi\times \phi)= S^1\times S^1$. Again, by item $(b)$ of Theorem \ref{teo4},  $\phi\times\phi$ would have the {\em POTP}. However  this is not possible because this flow is isometric and is not minimal \cite{Komuro84}. 
\end{example}

%\begin{remark}
%	The same ideas behind the proof of Theorem \ref{teo2.5} can be applied to show that the set of shadowable points of a homeomorphism is a $G_\delta$ set too.
%\end{remark}

Next, we study the relations between the set of shadowable points of a homeomorphism and the sets of shadowable points of its suspension flow.

Let $f:X\to X$ be a homeomorphism and $\tau:X\to(0,+\infty)$ be a continuous function. Consider the quotient space
$X^{\tau,f}=\{(x,t):0\leq t\leq \tau(x), x\in X\}/\sim$, where $(x,\tau(x))\sim(f(x),0)$ for all $x\in X$. The {\em suspension flow over $f$ with height function $\tau$} is the flow on $X^{\tau,f}$ defined by $\phi^{\tau,f}_t(x,s)= (x,s+t)$ whenever $s+t\in [0,\tau(x)]$. Replacing $d$ by the the equivalent metric if necessary, we can assume that $diam(X)=1$. Then, there is a natural metric $d^{\tau,f}$ on $X^{\tau,f}$ making it a compact metric space (this is the so-called {\em Bowen-Walters metric}, see \cite{Bw}).

Every suspension of $f$ is conjugate to the suspension of $f$ under the constant function $1$. A homeomorphism from $X^{1,f}$ to $X^{\tau,f}$ that conjugates the flows is given by the map $(x,t)\mapsto (x,t\tau(x))$.

\begin{theorem}\label{suspe1}
	If $\phi^{\tau,f}$ is the suspension flow of a homeomorphism $f$ on $X$ under a continuous map $\tau:X\rightarrow(0,+\infty)$, then
	$$Sh(\phi^{\tau,f})=(Sh(f)\times [0,1])/\sim.$$
\end{theorem}

With this theorem we have the following example of a flow whose shadowable set is non closed.

\begin{example}\label{examplechivo}
	Let $C$ be the usual ternary Cantor set in $[0,\,1]$ and $X = C \cup [1,\,2]$ with the usual metric of $\mathbb{R}$. Let $id\colon X\to X$ be the identity map of $X$. Example 2.1 in \cite{Morales16} shows that $Sh(id) = C\setminus\{1\}$, so by Theorem \ref{suspe1} $Sh(\phi^{1,f}) = (C\setminus\{1\})\times[0,1]/\sim$ which is a   proper subset non closed of $X^{1,f}$. 
\end{example}

Next, we will study the topological behavior of the shadowable points of $X$. We shall use the following standard topological concept. A subset of $X$ is a $G_\delta$ set if it is a countable intersection of open sets of $X$.
In \cite{Morales16},  examples of homeomorphisms where the set of shadowable points is a $G_\delta$ sets are given. We prove that this is always the case in the flow context:

\begin{theorem}\label{teo2.5}
The set of shadowable points of $\phi$ is a $G_{\delta}$ set of $X$.
\end{theorem}

Returning to the case of homeomorphisms, in \cite{Kawaguchi17}  Kawaguchi proved that the set of shadowable points of a homeomorphism is a  Borel set. But what he proved indeed is that such a set is a $F_{\sigma\delta}$ set of phase space, i.e., a countable intersection of countable union of closed sets. By making use of Theorem \ref{teo2.5}, we improve Kawaguchi's assertion by proving that the set of shadowable points is a $G_{\delta}$ set of the phase space:

\begin{corollary}\label{coroG}
The set of shadowable points of a homeomorphism $f\colon X\to X$ on a compact metric space $X$ is a $G_\delta$ set of $X$.
\end{corollary}

We denote by $Sh^+(\phi)$ the set of points such that given $\ep > 0$ there exists $\delta > 0$ such that every forward $(\delta,\,1)$-pseudo-orbit $(x_i,\,t_i)_{i=0}^\infty$ passing through $p$ can be $\ep$-shadowed. Each element of $Sh^+(\phi)$ is said forward shadowable point. Clearly, we have the relation $Sh(\phi) \subseteq Sh^+(\phi)$.

We recall that a chain transitive flow $\phi$ is one where $X$ is a chain transitive set. That is, for every $p,\,q\in X$ we have $p\sim q$, see \cite{Alongi07}. The following result  is a partial analogous of Theorem 1.1 in \cite{Kawaguchi17}  and characterizes the set of forward shadowable points for chain transitive flows.

\begin{theorem}\label{teo2.9}
If the flow $\phi$ is chain transitive then $Sh^+(\phi)= X$ or $Sh^+(\phi) = \emptyset$.
\end{theorem}

As a first consequence of the previous theorem, using suspensions we obtain an alternate proof of the following result, which follows from the fact that a map a chain transitive if and only if its suspension flow is (see \cite{Alongi07}): 

\begin{corollary}[Kawaguchi, N., \cite{Kawaguchi17}]
Every chain transitive homeomorphism either has the POTP or has no shadowable points.
\end{corollary} 

Recall that a transitive flow $\phi$, see \cite{CC}, is one for which there exists a point $x\in X$ such that $\omega(x) = X$
where 

\[\omega(x) = \{y\in X\colon y = \lim_{t_n\to+\infty}\phi_{t_n}(x)\mbox{ for some sequence } t_n \to +\infty\}.\]
In \cite{Morales16} the question of whether there exists a transitive homeomorphism with  non-trivial shadowable points set is posed. In \cite{Kawaguchi17} the author answers this question negatively and completely classifies the sets of shadowable points for chain transitive and transitive homeomorphisms. A well known result in topological dynamics states that every transitive flow is chain transitive \cite{Alongi07}.  The following corollary, which is an immediate consequence of Theorem \ref{teo2.9},  gives a partial negative answer to this question  in the flow case.

\begin{corollary}\label{coro4}
If $\phi$ is a transitive flow, then $Sh^+(\phi) = X$ or $Sh^+(\phi) = \emptyset$.
\end{corollary}

%Again through suspensions we can obtain the following result

%\begin{corollary}[Kawaguchi, N., \cite{Kawaguchi17}]
%Every transitive homeomorphism either has the POTP or has no shadowable points.
%\end{corollary} 

Corollary \ref{coro4} can be used to obtain information about the geometric Lorenz attractor \cite{Komuro85, Williams}. In \cite{Komuro85}, it is proved that if  $\phi$ is the geometric Lorenz attractor, then it does not have the finite forward {\em POTP} provided that its return map $f$ satisfies that $f(0)\neq 0$ or $f(1) \neq 1$. In this case we have  $Sh^+(\phi)\neq X$. It follows that $Sh^+(\phi)=\emptyset$.
So we obtain the following result. 
 
 \begin{corollary}\label{lorenz} The geometric Lorenz attractor does not have shadowable points.
 \end{corollary}
 
 We do not know if Theorem \ref{teo2.9} if we substitute forward shadowables points for shadowables points.  We have, however a related result. As stated previously, every transitive flow is chain transitive. However, it is well known the reciprocal is not true \cite{Alongi07}. The following theorem states that in the presence of forward shadowable points the two notions are equivalent.

\begin{theorem}\label{teo2.3.6}
	Let $X$ compact metric space and $\phi$ be a flow with shadowable points. Then $\phi$ is chain transitive if and only if is transitive.
\end{theorem}

\begin{example}
In virtue of Theorem \ref{teo2.3.6}, for every chain transitive flow $\phi$ which is not transitive we have $Sh(\phi) = \emptyset$. For a concrete example, take $X = S^1$ and $\phi$ the flow associated with the differential equation $\dot{\theta} = \sin^2(\theta)$ in the angular coordinates of $S^1$. 
\end{example}

In what follows we are going to see how Theorem  \ref{teo2.3.6} can be used to prove some interesting well known theorems.

Recall that a flow $\phi$ is distal if whenever $\inf_{t\in\mathbb{R}}d(\phi_t(x),\,\phi_t(y)) = 0$ implies $x=y$. Komuro showed that isometric flows with the pseudo orbit tracing property are minimal flows \cite{Komuro84}. Meanwhile Kato proved that equicontinuous flows with the pseudo orbit tracing property are also minimal flows \cite{Kato84}. Later,  He and Wang proved that distal flows with the pseudo-orbit tracing propery are minimal too \cite{He95}. Finally, Jiehua \cite{Jiehua1996} proved that pointwise recurrent flows with the pseudo-orbit tracing property are minimal.

In light of these results, is natural to ask if we still can conclude minimality if the pointwise recurrence hypothesis is weakened to suppose the flow to be chain recurrent. The answer is negative as there are nonminimal chain recurrent flows  with the pseudo orbit tracing property, for instance the suspension of the usual linear anosov map on the torus. However, there are not known examples of nontransitive chain recurrent flows with the pseudo-orbit tracing property on connected spaces. We are going to proof that transitivity is a necessary condition is the phase space is assumed to be connected and we give an example that shows that this is not the case on nonconnected phase spaces. Indeed, If $X$ is assumed to be connected and $\phi$ chain recurrent, then $\phi$ is necessarily chain transitive \cite{Alongi07}. The following corollary is immediate from Theorem \ref{teo2.3.6}:

\begin{corollary}\label{coroim}
    Let $\phi$ be a chain recurrent flow with shadowable points on a compact connected metric space $X$. Then $\phi$ is transitive.
\end{corollary}

The following example shows that the conclusion of Corollary \ref{coroim} cannot be guaranteed if we drop the connecteness hypothesis.
\begin{example}
	In \cite{mo}, it is proved that an equicontinuous homeomorphism $f$ on a compact metric space $X$  has the {\em POTP} if and only if $X$ is totally disconected. So if $C$ is the usual ternary Cantor set in the interval $[0,\,1]$, then identity map $id\colon C\to C$ has the {\em POTP}. The suspension flow of this map is then a chain recurrent flow that has {\em POTP} \cite{Thomas84}. But this flow is not transitive for its phase space is not connected. 
\end{example}

The following results mentioned previously can be obtained as a consequence of Corollary \ref{coroim}, giving thus, a more concise proof of these.

\begin{corollary}[He, L., Wang, M., \cite{He95}]
	Every distal flow with POTP on a connected compact metric is minimal.
\end{corollary}

\begin{proof}
It is enough to note that every distal flow is chain recurrent and every transitive distal flow is minimal.
\end{proof}

A flow $\phi$ is equicontinuous if the family of $t$-time maps $\{\phi_{t}\}_{t\in\mathbb{R}}$ is an equicontinuous family of homeomorphisms in $X$.

\begin{corollary}[Kato, K., \cite{Kato84}]
	Let $M$ be a Riemannian manifold and $\phi$ an equicontinuous flow with respect to the Riemannian metric of $M$. If $\phi$  has the finite POTP then $\phi$ is minimal.
\end{corollary}

\begin{proof}
It is enough to note that an equicontinuous flow is distal.
\end{proof}

\begin{corollary}[Komuro, M., \cite{Komuro84}]
	Let $M$ be a Riemannian manifold and $\phi$ an isometric flow with respect to the Riemannian metric of $M$. If $\phi$  has the finite POTP then $\phi$ is minimal.
\end{corollary}

\begin{proof}
Simply note that isometric flows are also distal.
\end{proof}

%%%%%%%%%%%%%%%%%%%%%%%%%%%%%%%%%%%%%%%%%%%%%%%%%%%%%%%%%%%%%%%%%%%%%%%%%%%%%

\section{Preliminaries} 

Let $X$ be a compact metric space. We say that a sequence $(x_n,t_n)_{n\in \mathbb{Z}}$ of $X\times\mathbb{R}$ is through some subset $K\subseteq X$  if $x_0\in K$ (see \cite{Morales16}). Now we introduce the following auxiliary definition.

\begin{definition}\label{defi5}
We say that a flow $\phi\colon X\times\mathbb{R}\to X$ has the {\em POTP} through a subset $K$, if given $\ep > 0$, there exists $\delta > 0$ such that every  $(\delta,\, 1)$-pseudo-orbit passing through $K$ can be $\ep$-shadowable.
\end{definition}

Note that we do not require the entire $(\delta,\, 1)$-pseudo-orbit to be contained in $K$, therefore the definition \ref{defi5} is stronger than the {\em POTP} \emph{on} $K$ \cite{Pilyugin2008}. 

A sequence of pairs $(x_i,\, t_i)_{i\in \mathbb{Z}}$ is a $(\delta,\,T_1,\,T_2)$-pseudo-orbit of $\phi$ if it is a $(\delta,\,T_1)$-pseudo-orbit of $\phi$ and satisfies $t_i \leq T_2$, for all $i\in\mathbb{Z}$. 

In \cite{Thomas84}, Thomas proved that a flow satisfies the {\em POTP} with respect to the parameter $T$ if and only if for every $\ep > 0$ we can find $\delta > 0$ such that every $(\delta,\,T,\, 2T)$-pseudo-orbit can be $\ep$-shadowed. We shall use the following lemma which is essentially contained in \cite{Thomas84}. We include its proof for the sake of completeness.

\begin{lemma}\label{lemma2}
Let $a>0$, $K\subseteq X$ and $\phi$ be a flow on a compact metric space $X$. Then the following statements are equivalent:
\begin{enumerate}[(1)]
\item For all $\ep>0$ there exists $\delta>0$ such that every $(\delta,a,2a)$-pseudo-orbit passing through $K$ is $\ep$-shadowed by an orbit of $\phi$.
\item $\phi$ has the POTP through $K$ with respect to the parameter $a$.
\item $\phi$ has the POTP through $K$.
\end{enumerate}
\end{lemma}
\begin{proof}
We assume that $a>1$. For the other case (when $a <1$) a similar argument can be used. Suppose that for all $\ep>0$ there exists $\delta>0$ such that every $(\delta,a,2a)$-pseudo-orbit passing through $K$ is $\ep$-shadowed by an orbit of $\phi$. First we prove that $\phi$ has the {\em POTP} through $K$ with respect to the parameter $a$ and then we prove that the flow has the  {\em POTP} through $K$. Let $(x_i,\,t_i)_{i\in\mathbb{Z}}$ be any $(\delta,\,a)$-pseudo-orbit of $\phi$ passing through $K$. For each $n\in \mathbb{Z}$, there exists $m_n\in \mathbb{N}$ such that $t_n=m_na+r_n$ with $a\leq r_n<2a$. Let $(s_n^m)_{n\in \mathbb{Z}}$ the sequence of sums associated to $m=(m_n)_{n\in \mathbb{Z}}$. Denote $A_n=s_n^m+n$ for all $n\in \mathbb{Z}$ and define the sequence $(y_i)_{i\in \mathbb{Z}}$ on $X$ such that $y_i=\phi_{a(i-A_n)}(x_n)$ if $A_n\leq i<A_{n+1}$. In addition, we define a sequence $\lambda=(\lambda_i)_{i\in \mathbb{Z}}$ of real numbers in the following way, for each $i\in \mathbb{Z}$, we set
$$
\lambda_i = \left\{
\begin{array}{cl}
a &\mbox{if \ } A_n\leq i<A_{n+1}-1,\\
r_n&\mbox{if\ \ } i=A_{n+1}-1.
\end{array}\right.
$$
Given $i\in\mathbb{Z}$ note that $a\leq \lambda_i<2a$ and let $n\in \mathbb{Z}$ be such that $A_n\leq i<A_{n+1}$. We have two cases.

{\em Case 1:} if $i<A_{n+1}-1$, then $$d(\phi_{\lambda_i}(y_i),y_{i+1})=d(\phi_a(\phi_{a(i-A_n)}(x_n)),\phi_{a(i+1-A_n)}(x_n))=0.$$

{\em Case 2:} if $i=A_{n+1}-1$, bearing in mind that $A_{n+1}-A_n=s_{n+1}^m-s_{n}^m+1=m_n+1$ we obtain
\begin{eqnarray*}
d(\phi_{\lambda_i}(y_i),y_{i+1})=d(\phi_{r_n}(\phi_{a(A_{n+1}-1-A_n)}(x_n)),x_{n+1})&=&d(\phi_{r_n}(\phi_{am_n}(x_n)),x_{n+1})\\ &=& d(\phi_{t_n}(x_n),x_{n+1})\leq \delta.
\end{eqnarray*}
That is, $(y_i,\lambda_i)_{i\in\mathbb{Z}}$ is a $(\delta,a,2a)$-pseudo-orbit of $\phi$ passing through $K$. Then, there are $z\in X$ and $h\in \mathrm{Rep}$ such that $d(\phi_{r-s^{\lambda}_n}(y_n),\phi_{h(r)}(z))\leq \ep$ where $s^{\lambda}_n\leq r<s^{\lambda}_{n+1}$ and $(s^{\lambda}_i)$ is the sequence of sums associated to $\lambda=(\lambda_i)_{i\in \mathbb{Z}}$. Let $w\in \mathbb{R}$ and $n\in \mathbb{Z}$ such that $s^{t}_n\leq w<s^{t}_{n+1}$, where $(s^{t}_n)$ is associated to $t=(t_i)_{i\in \mathbb{Z}}$. Since $s^{t}_n=s^{\lambda}_{A_n}$, then $s^{\lambda}_{A_n}\leq w<s^{\lambda}_{A_{n+1}}=s^{\lambda}_{A_n+m_n+1}$. Hence, there is $0\leq j\leq m_n$ such that $s^{\lambda}_{A_n+j}\leq w<s^{\lambda}_{A_{n}+j+1}$ and then
\begin{eqnarray*}
\ep \geq d(\phi_{w-s^{\lambda}_{A_n+j}}(y_{A_n+j}),\phi_{h(w)}(z))&=& d(\phi_{w-s^{t}_n}(\phi_{s^{t}_n-s^{\lambda}_{A_n+j}}(y_{A_n+j})),\phi_{h(w)}(z))\\
&=& d(\phi_{w-s^{t}_n}(\phi_{s^{t}_n-s^{\lambda}_{A_n+j}}(\phi_{aj}(x_n))),\phi_{h(w)}(z))\\
&=& d(\phi_{w-s^{t}_n}(x_n),\phi_{h(w)}(z)).
\end{eqnarray*}
It follows that $\phi$ has the {\em POTP} through $K$ with respect to the parameter $a$. Now we prove that the flow has the  {\em POTP} through $K$. Fix $m\in \mathbb{N}$ such that $m\geq a$. Given $\ep>0$ choose $\delta>0$ satisfying the following conditions:

\begin{enumerate}
\item Every $(\delta,a)$-pseudo-orbit passing through $K$ is $\frac{\ep}{2}$-shadowable.
\item For each $0\leq t\leq 2m$ we have $d(\phi_t(x),\phi_t(y))<\frac{\ep}{2}$, whenever $d(x,y)<\delta$.
\end{enumerate}
Let $0<\delta^{\prime}<\delta/m$ and take $0<\beta<\delta^{\prime}$ so that $d(x,y)<\beta$ implies that $d(\phi_t(x),\phi_t(y))<\delta^{\prime}$ for $0\leq t\leq 2m$. Let $(x_n,t_n)_{n\in \mathbb{Z}}$ be a $(\beta,1)$-pseudo-orbit for $\phi$ passing through $K$ with $1\leq t_n\leq 2$ for all $n\in \mathbb{Z}$. Consider the sequence of pairs $(x_{im},\lambda_i)_{i\in \mathbb{Z}}$ where $\lambda_i=\sum_{j=0}^{m-1}t_{j+im}$ for every $i\in\mathbb{Z}$. We denote $\lambda_i(k)=\sum_{j=k}^{m-1}t_{j+im}$ with $0\leq k<m$. Then

$$d(\phi_{\lambda_i}(x_{im}),x_{(i+1)m})\leq \sum_{r=1}^{m}d(\phi_{\lambda_i(r)}(\phi_{t_{im+r-1}}(x_{im+r-1})),\phi_{\lambda_i(r)}(x_{im+r}))\leq m\delta^{\prime}<\delta,$$ 
because $a\leq \lambda_i\leq 2m$. So, $(x_{im},\lambda_i)_{i\in \mathbb{Z}}$ is a $(\delta,a)$-pseudo-orbit for $\phi$ passing through $K$. Hence, there are $z\in X$ and $h\in \mathrm{Rep}$ such that $d(\phi_{t-s^{\lambda}_n}(x_{nm}),\phi_{h(t)}(z))\leq \frac{\ep}{2}$ where $s^{\lambda}_n\leq t< s^{\lambda}_{n+1}$. Now, for $0\leq k<m$ denote $s^{t}_k(r)=\sum_{j=r}^{k-1}t_j$ we have
$$d(\phi_{s^{t}_k}(x_0),x_k)\leq \sum_{r=1}^{k}d(\phi_{s^{t}_k(r)}(\phi_{t_{r-1}}(x_{r-1})),\phi_{s^{t}_k(r)}(x_r))<k\delta^{\prime}<\delta.$$ Then for $s^{t}_k\leq t<s^{t}_{k+1}$ $$d(\phi_{t-s^{t}_k}(x_k),\phi_{h(t)}(z))\leq d(\phi_{t-s^{t}_k}(x_k),\phi_{t-s^{t}_k}(\phi_{s^{t}_k}(x_0)))+d(\phi_{t}(x_0),\phi_{h(t)}(z))\leq \ep.$$
For $m\leq k<2m$, we continue in the same manner. So we will have that the orbit $(\phi_t(z))_{t\in\mathbb{R}}$  $\ep$-shadows the $(\beta,1,2)$-pseudo-orbit of $\phi$ passing through $K$. Applying the above reasoning we obtain that $(\phi_t(z))$,  can be $\ep$-shadowed by the $(\beta,1)$-pseudo-orbit of $\phi$ passing through $K$.
\end{proof}

Hence a flow $\phi$ in a compact metric space $X$ has {\em POTP}  if and only if for all $\ep>0$ there exists $\delta>0$ such that every $(\delta,1,2)$-pseudo-orbit is $\ep$-shadowed by an orbit of $\phi$.

We denote by $B[\,\cdot,\,\delta]$ the close ball operation on $X$. 

Clearly, if a flow has the {\em POTP} through a set $K$ then every point is $K$ is shadowable respectively. The reciprocal is also true in the compact case as shown in the following lemma. 

\begin{lemma}\label{lemma3}
Let $\phi$ be a flow on a compact metric space $X$. If every point of a compact subset $K$ of $X$ is shadowable, then $\phi$ has the POTP through the set $K$.
\end{lemma}
\begin{proof}
Assume by contradiction that there exists a nonempty compact subset $K$ such that every point in $K$ is shadowable but does not have the {\em POTP} through $K$. Then there is $\ep > 0$ and a sequence $(\xi^k)_{k\in\mathbb{N}}=(\xi_n^k,\,t_n^k)_{n\in\mathbb{Z}}$ of $(\frac{1}{k},\,1,\,2)$-pseudo-orbits passing through $K$ which cannot be $2\ep$-shadowed. Since $K$ and $[1,\,2]$ are compact, we can assume that $\xi_0^k\to p$ for some $p\in K$ and $t_0^k\to t_0$ for some time $t_0\in[1,\,2]$. We have that $p$ is shadowable, so for $\ep$ as above, we choose $\delta > 0$ from the shadowableness of $p$ with $\delta < \frac{\ep}{3}$. Since $X\times [0,\,2]$ is compact, $\phi|_{X\times [1,\,2]}$ is uniformly continuous and so our $\delta$ can also be chosen so that if $d((x,\, t),\,(y,\,s))\leq \delta$ with $0\leq s,\,t\leq 2$ then $d(\phi_t(x),\,\phi_s(y))\leq \frac{\ep}{3}$. We set a sequence $\hat{\xi}^k=(\hat{\xi}_n^k,\,\hat{t}_n^k)_{n\in\mathbb{Z}}$ as follows,
        
        \[
          \hat{\xi}^k =
            \begin{cases}
              (\xi^k_n,\, t_n^k), &\mbox{if } n\neq0,\\
              (p,\,t_0),        &\mbox{if } n = 0.
            \end{cases}
        \]
Clearly all such sequences are passing through $p$. Moreover,
     \[
          d(\phi_{\hat{t}_n^k}(\hat{\xi}_n^k),\, \hat{\xi}_{n+1}^k) =
            \begin{cases}
              d(\phi_{t_n^k}(\xi_n^k),\, \xi_{n+1}^k), &\mbox{if } n\neq0,\,-1,\\
              d(\phi_{t_0}(p),\,\xi_1^k),        &\mbox{if } n = 0,\\
              d(\phi_{t_{-1}^k}(\xi_1^k),\,p),    &\mbox{if } n = -1,
            \end{cases}
        \]
so   
       \[
          d(\phi_{\hat{t}_n^k}(\hat{\xi}_n^k),\, \hat{\xi}_{n+1}^k) \leq
            \begin{cases}
              \frac{1}{k}, &\mbox{if } n\neq0,\,-1,\\
              d(\phi_{t_0}(p),\,\phi_{t_0^k}(\xi_0^k))+\frac{1}{k},        &\mbox{if } n = 0,\\
              d(\xi_{0}^k,\,p)+\frac{1}{k},    &\mbox{if } n = -1.
            \end{cases}
        \]       
As $\phi$ is continuous and $(\xi_0^k,\, t_0^k)\to (\xi_0,\,p)$ we obtain that $(\hat{\xi}_n^k)$ is a $(\delta,\,1,\,2)$-pseudo-orbit for $k$ large. Then for such $k$ it follows that there are $x_k\in X$ and $h\in \mathrm{Rep}$  such that $d(p\star t,\,\phi_{h(t)}(x_k))\leq\ep$ for all $t\in\mathbb{R}$.  For the sequences $(\hat{t}_i^k)_{k\in\mathbb{Z}}$ and $(t_i^k)_{k\in\mathbb{Z}}$ we write
       
       \[\hat{s}_i^k =
      \begin{cases}
        \displaystyle{\sum_{j=0}^{i-1}} \hat{t}_j^k & i > 0,\\
		    0				             & i = 0, \\
		    - \displaystyle{\sum_{j=i}^{-1}} \hat{t}_j^k & i < 0,
      \end{cases}
    \]
and 
   \[s_i^k =
      \begin{cases}
         \displaystyle{\sum_{j=0}^{i-1}} t_j^k & i > 0,\\
		    0				             & i = 0, \\
		    - \displaystyle{\sum_{j=i}^{-1}} t_j^k & i < 0.
      \end{cases}
    \]
We will consider the three possible cases: $t_0^k < \hat{t}_0^k$ , $t_0^k = \hat{t}_0^k$ and $t_0^k > \hat{t}_0^k$. Note that in every case we have $|s_i^k - \hat{s}_i^k| = |t_0^k - \hat{t}_0^k|$ for all $i\in\mathbb{Z}$. We consider only the $t_0^k < \hat{t}_0^k$ case being the other two cases analogous. Then $s_i^k < \hat{s}_i^k$ for all $i\in\mathbb{Z}$. Let $t\in \mathbb{R}$ and let $i\in\mathbb{Z}$ such that $s_i^k\leq  t < s_{i+1}^k$. We have two cases:

\emph{Case 1:} if $\hat{s}_i^k\leq  t < s_{i+1}^k$,  then in particular $\hat{s}_i^k\leq  t < \hat{s}_{i+1}^k$, so
    \begin{align*}
    d(\phi_{t-s_i^k}(\xi_i^k),\,\phi_{h(t)}(x_k))&\leq d(\phi_{t-s_i^k}(\xi_i^k),\,\phi_{t-\hat{s}_i^k}(\hat{\xi}_i^k)) + d(\phi_{t-\hat{s}_i^k}(\hat{\xi}_i^k),\,\phi_{h(t)}(x_k))\\
 &\leq\frac{\ep}{3} + \ep < 2\ep.
\end{align*}  

\emph{Case 2:} if $s_i^k\leq t < \hat{s}_i^k$, again in particular, $\hat{s}_{i-1}^k\leq t < \hat{s}_i^k$, so
    \begin{align*}
    d(\phi_{t-s_i^k}&(\xi_i^k),\,\phi_{h(t)}(x_k)) \leq  d(\phi_{t-s_i^k}(\xi_i^k),\,\phi_{t-\hat{s}_{i-1}^{k}}(\hat{\xi}_{i-1}^{k})) + d(\phi_{t-\hat{s}_{i-1}^{k}}(\hat{\xi}_{i-1}^{k}),\,\phi_{h(t)}(x_k))\\
    &\leq  d(\phi_{t-s_i^k}(\xi_i^k),\,\hat{\xi}_i^k) + d(\hat{\xi}_i^k,\,\phi_{\hat{t}_{i-1}^k}(\hat{\xi}_{i-1}^k)) + d(\phi_{\hat{t}_{i-1}^k}(\hat{\xi}_{i-1}^k),\,\phi_{t-\hat{s}_i^{k-1}}(\hat{\xi}_{i-1}^{k})) + \ep\\
    &\leq \frac{\ep}{3} + \frac{\ep}{3} + \frac{\ep}{3} + \ep = 2\ep,
    \end{align*}
thus $d(\phi_{t-s_i^k}(\xi_i^k),\,\phi_{h(t)}(x_k))\leq 2\ep$ for all $s_i^k\leq  t < s_{i+1}^k$. It follows that $\xi^k$ can be $2\ep$-shadowed, which is a contradiction. This proves the result. 
    \end{proof}

\begin{lemma}\label{lemma1}
Given  a flow $\phi$ in a compact metric space $(X,\, d)$, then  $Sh(\phi)$ is invariant. So, if not empty, it is a union of orbits of $\phi$.
\end{lemma}
\begin{proof}
Let $x$ be a shadowable point of $X$. Let $\ep > 0$ and $s\in\mathbb{R}$ given. Since $\phi_s$ is uniformly continuous, so we can choose $0 <\ep' < \ep$ such that whenever $d(x,\,y) < \ep'$ we have $d(\phi_s(x),\,\phi_s(y))<\ep$. For $\ep'$, let  $\delta>0$ such that any  $(\delta,\, 1)$-pseudo-orbit passing through $x$ can be $\ep'$-shadowed. Similarly, $\phi_{-s}$ is uniformly continuous so we can choose $\delta' > 0$ with the property that $d(\phi_{-s}(x),\, \phi_{-s}(y))\leq \delta$ whenever $d(x,\, y) < \delta'$. Now let $(x_i,\, t_i)_{i\in \mathbb{Z}}$ be a $(\delta',\, 1)$-pseudo-orbit passing through $\phi_s(x)$. Because $d(\phi_{\,t_i}(x_i),\, x_{i+1}) \leq \delta'$ we have by the choice of $\delta'$ that \[d(\phi_{-s}(\phi_{t_i}(x_i)),\, \phi_{-s}(x_{i+1})) = d(\phi_{t_i}(\phi_{-s}(x_i)),\, \phi_{-s}(x_{i+1})) \leq \delta\] and hence $(\phi_{-s}(x_i), t_i)_{i\in \mathbb{Z}}$ is a $(\delta,\, 1)$-pseudo-orbit passing through $x$. By definition, there are $h\in \mathrm{Rep}$ and $y\in X$ such that
\[d(x\star t, \phi_{h(t)}(y)) \leq \ep',\ \mbox{for every}\ t\in\mathbb{R}.\]
Then, if $s_i\leq t< s_{i+1}$, it follows that $d(\phi_{t-s_i}(\phi_{-s}(x_i)),\,\phi_{h(t)}(y))\leq\ep'$ for every $t\in\mathbb{R}$ which implies $d(\phi_{t-s_i}(x_i), \phi_{h(t)}(\phi_s(y)))\leq \ep$. Therefore, $d(x_0\star t, \phi_{h(t)}(\phi_s(y)))\leq \ep$ for each $t\in \mathbb{R}$. Thus, every $(\delta',\,1)$-orbit passing through $\phi_s(x)$ can be $\ep$-shadowed by a point in $X$. This completes the proof.
\end{proof}

\begin{lemma}\label{lemma4}
If $\phi$ is a flow on a compact metric space $X$, then $Sh(\phi)\cap CR(\phi) \subseteq \Omega(\phi)$.
\end{lemma}
\begin{proof}
Let $p\in Sh(\phi)\cap CR(\phi)$  and $\ep > 0$ be given. Then there exists $\delta > 0$ from the shadowableness of $p$. Since $p$ is a chain recurrent point, there exists a $(\delta,\,1)$-chain $(x_i,\, t_i)_{i = 0}^k$ with $p = x_0 = x_k$. For every integer number $n$ we put $x_{kn+i} = x_i$, $t_{kn +i} = t_i$ for $0\leq i < k$.  So, $(x_i,\, t_i)_{i\in\mathbb{Z}}$ is a $(\delta,\,1)$-pseudo-orbit for $\phi$ and therefore there are $y\in X$ and $g\in\mathrm{Rep}$  such that $d(p\star t,\,\phi_{g(t)}(y))\leq\ep$ for every $t\in \mathbb{R}$. It follows that $y\in B[p,\,\ep]$ because $g(0) = 0$ by definition. For every $j \geq 0$ make $m_j = j\sum_{i=0}^{k-1}t_i$. Then,
\begin{align*}
d(x,\,\phi_{g(m_j)}(y)) &= d(x\star m_j, \phi_{g(m_j)}(y))\\
						&\leq\ep,\ \forall\,j\geq0
\end{align*}
and  $m_j\geq jk$ for all $j\geq 0$. So $m_j\to\infty$ ($j\to\infty$). Since $g\in\mathrm{Rep}$, $g(m_j)\to\infty$. Therefore $x\in\Omega(\phi)$, and the lemma follows.
\end{proof}

\begin{lemma}\label{3.55}
If $\phi$ is a expansive flow on a compact metric space $X$, then $CR(\phi)\cap Sh(\phi)\subset \overline{Per(\phi)}$.
\end{lemma}
\begin{proof}
Without loss of generality, by expansiveness, we can suppose that the flow has no singularities. Let $x\in CR(\phi)\cap Sh(\phi)$ and $\varepsilon\in (0,1)$. We can consider that $\varepsilon$ satisfies Lemma 3.10 in \cite{Hel} with respect to $\lambda=\frac{1}{2}$. Take $\delta>0$ satisfying the definition of shadowing respect to $\varepsilon$. Since $x\in CR(\phi)$, there is a $(\frac{\delta}{2},\, 3)$-chain $(x_i,\,t_i)_{i=1}^{k}$ where $x_0=x_k=x$ and $t_i\geq 3$. Assume that $\delta$ comes from expansivity with respect to $\varepsilon$. Extend the $(\frac{\delta}{2},\, 3)$-chain $(x_i,\,t_i)_{i=1}^{k}$ to a $(\frac{\delta}{2},\, 3)$-pseudo-orbit. Thus, there are $z\in X$ and $\alpha\in\mathrm{Rep}$ such that $d(\phi_{\alpha(t)}(z),\,\phi_{t-s_i}(x_i))\leq \frac{\delta}{2}$ for $s_i\leq t<s_{i+1}$. If $L=t_0+\ldots+t_{k-1}$, then $d(\phi_{\alpha(t+L)}(z),\,\phi_{t-s_i}(x_i))\leq \frac{\delta}{2}$ for $s_i\leq t<s_{i+1}$. Therefore
\begin{center}
$d(\phi_{\alpha(t+L)}(z),\phi_{\alpha(t)}(z)\leq \delta$ for every $t\in \mathbb{R}$. 
\end{center}
Take $u=\alpha(t)$, then
$$d(\phi_{\alpha(\alpha^{-1}(u)+L)}(z),\,\phi_{u}(z))=d(\phi_{\alpha(\alpha^{-1}(u)+L)-\alpha(L)}(\phi_{\alpha(L)}(z)),\,\phi_{u}(z))\leq \delta$$ for every $u\in \mathbb{R}$, where $\alpha(\alpha^{-1}(u)+L)-\alpha(L)\in \mathrm{Rep}$. Hence $\phi_{\alpha(L)}(z)\in \phi_{(-\varepsilon,\varepsilon)}(z)$. 

Moreover since $d(\phi_{\alpha(t)}(z),\,\phi_{t}(x))\leq \frac{\varepsilon}{2}$ for $0\leq t<t_0$, then $\frac{1}{2}s\leq \alpha(s)$  for some $2\leq s\leq t_0$, by Lemma 3.10 in \cite{Hel}. Then $\varepsilon\leq \alpha(s)\leq\alpha(L)$ since $s\leq L$. Therefore $z\in Per(\phi)$. 
\end{proof}

Now we introduce another auxiliary definition.

\begin{definition}\label{3.6}
Given $\ep>0$ and $Z$ a subset of $X$. We say that a flow $\phi$ has the POTP through a subset $K$ if there exists $\delta > 0$ such that every  $(\delta,\, 1)$-pseudo-orbit passing through $K$ can be $\ep$-shadowable.
\end{definition}

\begin{lemma}\label{lemma3.7}
Let $\phi$ be a flow  on the compact metric space $X$ and let $\ep>0$. If the flow $\phi$ has the POTP through a compact subset $K$, then there is $\delta > 0$ so that $\phi$ has the $2\ep$-POTP through $B[K,\,\delta]$.
\end{lemma}

\begin{proof}
Suppose by contradiction that a flow $\phi$ has the {\em POTP} through a subset $K$ but for every $\delta > 0$, we can find a $(\delta,\,1)$-pseudo-orbit passing through $B[K,\,\delta]$ that cannot be $2\ep$-shadowed.

Take a $\delta>0$ from the {\em POTP} through $K$ with $\delta <\ep$, and let $(\xi^k)_{k\in\mathbb{N}}$ be a sequence of $(\frac{1}{k},\, 1)$-pseudo-orbits passing through $B[K,\,\frac{1}{k}]$ which cannot be $2\ep$-shadowed. For every $k\in\mathbb{N}$ we write $\xi^k = (\xi_n^k,\, t_n^k)_{n\in\mathbb{Z}}$. It follows from the definition that there is a sequence $x_k\in K$ such that $d(\xi_0^k, x_k)\leq\frac{1}{k}$ for all $k\in\mathbb{N}$. Since $X$ is compact, the flow $\phi$ is uniformly continuous in \sloppy $X\times [-t^1_0,\,t^1_0]$, so we can choose $k$ with the property that \[\max\{\max_{-t_0^1\leq t\leq t_0^1}\{d(\phi_t(\xi_0^k),\,\phi_t(x_k))\},\,\tfrac{1}{k}\}\leq\frac{\delta}{2}.\]
Fix $k$ and define a sequence ${\xi} = ({\xi}_n,\,t_n)_{n\in\mathbb{Z}}$ by
\[
({\xi}_n,\,t_n) = 
\begin{cases}
(\xi_n^k,\,t_n^k) &\mbox{if } n\neq0,\\
(x_k,\, t_0^k) &\mbox{otherwise}.
\end{cases}
\]
Clearly, $d(\phi_{t_n}(\xi_n),\,\xi_{n+1})\leq\frac{1}{k} < \delta$ for $n\neq -1,\,0$. Since

\[
d(\phi_{t_{-1}}(\xi_{-1}),\,\xi_0) = d(\phi_{t_{-1}^k}(\xi_{-1}^k),\,x_k)\leq
			d(\phi_{t_{-1}^k}(\xi_{-1}^k),\, \xi_0^k) + d(x_k,\,\xi_0^k)\leq
			\tfrac{1}{k}+\tfrac{1}{k}=\frac{2}{k}\leq\delta
\]
and 

\[
d(\phi_{t_0}(\xi_0),\,\xi_1) = d(\phi_{t_0^1}(x_k),\,\xi_1^k)\leq
		d(\phi_{t_0^1}(x_k),\,\phi_{t_0^1}(\xi_0^k))+ d(\phi_{t_0^1}(\xi_0^k),\,\xi_1^k)
		\leq \frac{\delta}{2}+\frac{\delta}{2}=\delta,
\]
we see that $\xi$ is a $(\delta,\, 1)$-pseudo-orbit. Since $\xi_0 = x_k\in K$ by definition, we obtain that $\xi$ can be $\ep$-shadowed by a point $y\in X$. Thus, there exists $h\in \mathrm{Rep}$ such that $$d(\xi_0\star t,\,\phi_{h(t)}(y))<\ep,\quad \mbox{for each}\ t\in \mathbb{R}.$$
Note that for $i\neq 0$ and $t$ such that $s_i\leq t < s_{i+1}$, we have $$\xi_0\star t=\phi_{t-s_i}(\xi_i)=\phi_{t-s_i}(\xi_i^k)=\xi^k_0\star t.$$ 
Hence $\xi_0\star t = \xi_0^k\star t$ for $t\not \in [s_0,\, s_{1})$. Furthermore, for $t\in[s_0,\,s_1)$, 
\begin{align*}
d(\xi_0^k\star t,\,\phi_{h(t)}(y))
	&\leq d(\xi_0^k\star t,\,\xi_0\star t) + d(\xi_0\star t),\,\phi_{h(t)}(y)) \\
	&\leq d(\phi_t(\xi_0^k),\,\phi_t(\xi_0)) + d(\xi_0\star t,\,\phi_{h(t)}(y))\\
	&\leq\frac{\delta}{2} + \ep\\
	&\leq 2\ep.
\end{align*}
Thus, $d(\xi_0^k\star t,\,\phi_{h(t)}(y))\leq 2\ep$ for all $t\in \mathbb{R}$. It follows that $\xi^k$ is $2\ep$-shadowed, which is a contradiction. This proves the result.

\end{proof}

%%%%%%%%%%%%%%%%%%%%%%%%%%%%%%%%%%%%%%%%%%%%%%%%%%%%%%%%%%%%%%%%%%%%%%%%%%

\section{Proofs}

\begin{proof}[Proof of Theorem \ref{teo4}]
To prove Item (a), by Lemma \ref{lemma1} we have that $Sh(\phi)$ is invariant, example \ref{examplechivo} shows that this set can be noncompact and Corollary \ref{lorenz} shows that it can be empty.
Item (b) follows by making $K = X$ in Lemma \ref{lemma3} for  we that $\phi$ has the {\em POTP} if and only if $Sh(\phi) = X$. 

Item (c) follows since $\Omega(\phi)\subset CR(\phi)$, then we have that if $CR(\phi)\subseteq Sh(\phi)$, then $\Omega(\phi) =  CR(\phi)$ by Lemma \ref{lemma4}. 

Similarly, Item (d) follows since $Per(\phi)\subseteq CR(\phi)$, then $CR(\phi)=\overline{Per(\phi)}$ by Lemma \ref{3.55}. 

To prove Item (e), let $f\colon X\to Y$ be an equivalence between $\phi$ and $\psi$ on the metric spaces $(X,\,d_x)$ and $(Y,\,d_y)$ respectively. Suppose that for each $x\in X$ there exists $h_x\in \mathrm{Rep}$ such that $$f^{-1}(\psi(f(x),\,h_x(t)))=\phi(x,\,t),\ \mbox{for each}\ t\in \mathbb{R}.$$ Let $a=\min \{ h_x(1): x\in X\}$. By compactness of $X$ such $a$ exists and indeed is positive. Now, given $\ep>0$, choose $\ep^{\prime}>0$ such that $d_y(y_1,y_2)<\ep^{\prime}$ implies $d_x(f^{-1}(y_1),f^{-1}(y_2))<\ep$ for every $y_1,y_2\in Y$. Suppose $p\in f^{-1}(Sh(\psi))$. By Lemma \ref{lemma2}, there exists $\delta^{\prime}>0$ such that each $(\delta^{\prime},a)$-pseudo-orbit passing through $f(p)$ can be $\ep^{\prime}$-shadowed by an orbit of $\psi$. Also choose $\delta>0$ so that $d_y(f(x_1),f(x_2))<\delta^{\prime}$ whenever $d_x(x_1,x_2)<\delta$ for all $x_1,x_2\in X$. Now let $(x_n,t_n)_{n\in\mathbb{Z}}$ be a $(\delta,1)$-pseudo-orbit for $\phi$ passing through $p$. Then $d_y(f(\phi_{t_n}(x_n)),f(x_{n+1}))<\delta^{\prime}$. By definition of conjugacy we have $$d_y(\psi_{h_{x_n}(t)}(f(x_n)),f(x_{n+1}))<\delta^{\prime},\quad \mbox{for every}\ n\in\mathbb{Z}.$$
Consider the sequence $(f(x_n),h_{x_n}(t_n))_{n\in\mathbb{Z}}$. Since $t_n\geq 1$ it follows that $h_{x_n}(t_n)\geq a$ for all $n\in \mathbb{Z}$. So $(f(x_n),h_{x_n}(t_n))_{n\in\mathbb{Z}}$ is a $(\delta^{\prime},a)$-pseudo-orbit for $\psi$ passing through $f(p)$. Then there are $y=f(z)$ in $Y$ and $\alpha \in \mathrm{Rep}$ such that $$d_y(f(x_0)\star t, \psi_{\alpha(t)}(y))<\ep^{\prime},\ \mbox{for all}\ t\in\mathbb{R}.$$
It follows that
\begin{equation}\label{jojo}
d_x\left(f^{-1}(f(x_0)\star t), f^{-1}(\psi_{\alpha(t)}(f(z)))\right)<\ep,\quad \mbox{for all}\ t\in\mathbb{R}.  
\end{equation}
Fix $t\in\mathbb{R}$ and $n\in \mathbb{Z}$ such that $s_n \leq t<s_{n+1}$. Since $h_{x_n}\in\mathrm{Rep}$ we have $0\leq h_{x_n}(t-s_n)<h_{x_n}(t_n)$. Therefore, if we denote $\widehat{s}_n=\sum^{n-1}_{j=0} h_{x_{j}}(t_{j})$ for $n\geq 1$ and $\widehat{s}_n=-\sum^{-1}_{j=n} h_{x_{j}}(t_{j})$ for $n\leq 0$ we have $\widehat{s}_n\leq h_{x_n}(t-s_n)+\widehat{s}_n<\widehat{s}_{n+1}$. Take $\widehat{t}=h_{x_n}(t-s_n)+\widehat{s}_n$, so $\psi^{\widehat{t}}(f(x_0))=\psi_{h_{x_n}(t-s_n)}(f(x_n))$. By (\ref{jojo}) it follows for $t=\widehat{t}$ that
$$d_x(\phi_{t-s_n}(x_n),\phi_{h_{z}^{-1}(\alpha(\,\widehat{t}\ ))}(z))=d_x(f^{-1}(\psi_{\,\widehat{t}-\widehat{s}_n}(f(x_n))), f^{-1}(\psi_{\alpha(\,\widehat{t}\,\,)}(f(z))))<\ep.$$
This implies that if we define $\widehat{\alpha}(t)=h_{z}^{-1}(\alpha(h_{x_n}(t-s_n)+\widehat{s}_n))$ we have
$$d_x(x_0\star t,\phi_{\widehat{\alpha}(t)}(z))<\ep, \quad \mbox{for every}\  t\in\mathbb{R}.$$
Since $t\mapsto h_{x_n}(t-s_n)+\widehat{s}_n$ is increasing, then $\widehat{\alpha}\in \mathrm{Rep}$. Therefore, $f^{-1}(Sh(\psi))\subseteq Sh(\phi)$. The inclusion $f(Sh(\phi))\subseteq Sh(\psi)$ is obtained analogously considering the equivalence $f^{-1}\colon Y\to X$. This completes the proof.
\end{proof}

\begin{proof}[Proof of Theorem \ref{suspe1}]
	We can assume without loss of generality that $\tau\equiv 1$. Given $(z,t)\in Sh(\phi^{1,f})$ since $Sh(\phi^{1,f})$ is invariant by $\phi^{1,f}$, then $(z,\frac{1}{2})\in Sh(\phi^{1,f})$. Let $\ep>0$ be given. Choose $\ep^{\prime}>0$ with $\ep^{\prime}<\min\{\ep,\frac{1}{4}\}$ so that $d(f^i(x),f^i(y))<\ep$ for $i=-1,0,1$, whenever $d(x,y)<\ep^{\prime}$. Choose $\delta>0$ from the definition of shadowable point for $\phi^{1,f}$ with respect to $\ep^{\prime}$. Also take $0<\delta^{\prime}<\delta$ so that $d(x,y)<\delta^{\prime}$ implies $d(f(x),f(y))<\delta$. Let $\{ x_n\}_{n\in \mathbb{Z}}$ be any $\delta^{\prime}$-pseudo-orbit of $f$ with $x_0=z$. Consider the pair of sequences $(x_n,\frac{1}{2})_{n\in \mathbb{Z}}$ and $(t_n)_{n\in \mathbb{Z}}$ such that $t_n=1$ for each $n\in \mathbb{Z}$. Then  
	\begin{eqnarray*}
		d^{1,f}(\phi^{1,f}_{t_n}(x_n,\tfrac{1}{2}),(x_{n+1},\tfrac{1}{2}))&=& d^{1,f}((f(x_n),\tfrac{1}{2}),(x_{n+1},\tfrac{1}{2}))\\
		&=& \tfrac{1}{2}d(f(x_{n}),x_{n+1})+\tfrac{1}{2}d(f^{2}(x_{n}),f(x_{n+1})) \leq \delta.
	\end{eqnarray*}
	That is $((x_n,\tfrac{1}{2}),t_n)_{n\in \mathbb{Z}}$ is a $(\delta,\, 1)$-pseudo-orbit of $\phi^{1,f}$ with $x_0=z$. So, there are $(x,s)\in X^{1,f}$ and $\alpha\in \mathrm{Rep}$ such that 
	$$d^{1,f}(\phi^{1,f}_{\alpha(t)}(x,s),\phi^{1,f}_{t-n}(x_n,\tfrac{1}{2}))<\ep^{\prime},\quad \mbox{for}\ n\leq t<n+1\quad (n\in \mathbb{Z}).$$
	Now as $t=0$, we have $$d^{1,f}((x,s),(z,\tfrac{1}{2}))<\frac{1}{4},$$ so $\vert s-\frac{1}{2}\vert<\frac{1}{4}$. Moreover, since $d^{1,f}(\phi^{1,f}_{\alpha(t)}(x,s),\phi^{1,f}_{t}(z,\frac{1}{2}))<\ep^{\prime}<\frac{1}{4}$ for all $0\leq t<1$, it follows that $d^{1,f}((x,s+\alpha(1)),(z,\frac{3}{2}))<\frac{1}{4}$. Thus we obtain  $$\vert \tfrac{3}{2}-s-\alpha(1)\vert<\tfrac{1}{4}.$$
	Then $1\leq s+\alpha(1)<2$ and so $\phi^{1,f}_{\alpha(1)}(x,s)$ should be represented as $(f(x),s^{(1)})$ where $0\leq s^{(1)}<1$. Also we have $d^{1,f}(\phi^{1,f}_{\alpha(t)}(x,s),\phi^{1,f}_{t-1}(x_1,\frac{1}{2}))<\ep^{\prime}<\frac{1}{4}$ for all $1\leq t<2$. Thus $d^{1,f}((f(x),s^{(1)}+\alpha(2)-\alpha(1)),(x_1,\frac{3}{2}))<\frac{1}{4}$, therefore $\phi^{1,f}_{\alpha(2)}(x,s)$ should be represented as $(f^2(x),s^{(2)})$ where $0\leq s^{(2)}<1$. If we carry on in the same manner we will have that $\phi^{1,f}_{\alpha(n)}(x,s)$ should be represented as $(f^n(x),s^{(n)})$ where $0\leq s^{(n)}<1$ for each $n\in \mathbb{Z}$.
	For $t=n$, we have $$d^{1,f}(\phi^{1,f}_{n}(x,s),(x_n,\tfrac{1}{2}))=d^{1,f}((f^n(x),s^{(n)}),(x_n,\tfrac{1}{2}))<\ep^{\prime}.$$ If $f^n(x)=x_n$,  $d(f^n(x),x_n)<\ep$ is trivial. If $f^n(x)\neq x_n$, it follows that 
	$$\tfrac{1}{2}d(f^n(x),x_n)+\tfrac{1}{2}d(f^{n+1}(x),f(x_n))\leq d^{1,f}((f^n(x),s^{(n)}),(x_n,\tfrac{1}{2}))<\ep^{\prime}.$$
	Hence $d(f^n(x),x_n)<\ep^{\prime}$ or $d(f^{n+1}(x),f(x_n))<\ep^{\prime}.$ From the way we chose $\ep^{\prime}$ this implies that $d(f^{n}(x),x_n)<\ep$ for every $n\in \mathbb{Z}$. Therefore $z\in Sh(f)$.
	
	Conversely, let $z\in Sh(f)$ and $r\in [0,1]$. Given $\ep>0$, take $0<\ep^{\prime}<\ep$ so that $d(x,y)<\ep^{\prime}$ implies $d(f^{i}(x),f^{i}(y))<\tfrac{1}{2}\ep$ for $i=0,1,2$. Let $\delta$, with $0<\delta<\tfrac{1}{2}\ep^{\prime}$, from the definition of shadowable point for $f$ respect to $\ep^{\prime}$. Take $0<\delta^{\prime}<\min\{\tfrac{1}{4},\delta\}$ as in Lemma 2.5 in \cite{Thomas84} and $((x_k,s_k),(t_k))_{k\in\mathbb{Z}}$ a $(\delta^{\prime},2,4)$-pseudo orbit passing through $(z,r)$ for the suspension flow $\phi^{1,f}$ on $X^{1,f}$. Let $w_k=[s_k+t_k]$ denote the integer part of $s_k+t_k$. Hence
	$$d^{1,f}((f^{w_k}(x_k),s_k+t_k-w_k), (x_{k+1},s_{k+1}))<\delta^{\prime}\quad \mbox{for all}\ k\in \mathbb{Z}.$$
	Since $\delta^{\prime}<\tfrac{1}{4}$, by Lemma 2.4 in \cite{Thomas84}, we have that $\vert s_k+t_k-w_k-s_{k+1}\vert<\delta^{\prime}$ or $\vert 1+s_k+t_k-w_k-s_{k+1}\vert<\delta^{\prime}$ or $\vert 1+s_{k+1}+w_k-t_k-s_{k}\vert<\delta^{\prime}$. Now, let $n_k$ be a positive integer defined as follows 
	$$
	n_k = \left\{
	\begin{array}{cl}
	w_k &\mbox{if } \vert s_k+t_k-w_k-s_{k+1}\vert<\delta^{\prime},\\
	w_k-1 &\mbox{if } \vert 1+s_k+t_k-w_k-s_{k+1}\vert<\delta^{\prime},\\
	w_k+1 &\mbox{if } \vert 1+s_{k+1}+w_k-t_k-s_{k}\vert<\delta^{\prime}.
	\end{array}\right.
	$$
	Then by Lemma 2.5 in \cite{Thomas84} we obtain that $d(f^{n_k}(x_k),x_{k+1})<\delta$ for all $k\in \mathbb{Z}$. Define a sequence $(y_i)_{i\in\mathbb{Z}}$ in $X$ as follows:
	$$ y_i=f^{i-N_k}(x_k)\quad \mbox{for}\ N_k\leq i<N_{k+1},
	$$
	where $(N_k)_{k\in \mathbb{Z}}$ is the sequence of sums associated to $(n_k)_{k\in \mathbb{Z}}$. Obviously this sequence is a $\delta$-pseudo orbit of $f$ passing through $z$. Hence, there exists $x\in X$ such that $d(f^i(x),y_i)<\ep^{\prime}$ for every $i\in \mathbb{Z}$. Therefore we get 
	\begin{equation}\label{somb}
	d(f^{j+N_k}(x),f^{j}(x_k))<\ep^{\prime}\quad \mbox{for}\ 0\leq j<n_k\ (k\in \mathbb{Z}).
	\end{equation}
	Now, take the point $(x,t)\in X^{1,f}$ and define $\alpha:\mathbb{R}\rightarrow\mathbb{R}$ in the following way.
	$$\alpha(t)=\frac{s_{k+1}+n_k-s_k}{t_k}(t-T_k)+s_k+N_k-s_0\quad\mbox{for}\ T_k\leq t<T_{k+1},$$
	where $(T_k)_{k\in \mathbb{Z}}$ is the sequence of sums associated to $(t_k)_{k\in \mathbb{Z}}$. It is clear that $\alpha$ is continuous with $\alpha(0)=0$. Moreover, since $n_k\geq 1$ then $\alpha \in \mathrm{Rep}$. We claim that $\phi^{1,f}_{\mathbb{R}}(x,r)$ is an orbit on $X^{1,f}$ which $\ep$-traces $((x_k,s_k),(t_k))_{k\in\mathbb{Z}}$. Let $t\in \mathbb{R}$ and let $k\in\mathbb{Z}$ be such that $T_k\leq t<T_{k+1}$ we get
	\begin{eqnarray*}
		\vert \alpha(t)-s_k-N_k+s_0-(t-T_k)\vert &=&\left\vert \frac{s_{k+1}+n_k-s_k-t_k}{t_k}(t-T_k)\right\vert\\
		&=& \vert s_{k+1}+n_k-s_k-t_k\vert\left\vert\frac{t-T_k}{t_k}\right\vert.
	\end{eqnarray*}
	Since $\vert s_k+t_k-n_k-s_{k+1}\vert<\delta^{\prime}$ and $0\leq t-T_k<t_k$, we have 
	\begin{equation}
	\vert \alpha(t)-s_k-N_k+s_0-(t-T_k)\vert<\delta^{\prime}.
	\end{equation}
	Now if $j$ is a positive integer which makes $0\leq s_k+t-T_k-j<1$, then 
	$0\leq j\leq s_k+t_k\leq n_k+2$. So by (\ref{somb}) and the choice of $\ep^{\prime}$ we get $d(f^{j+N_k}(x),f^{j}(x_k))<\tfrac{1}{2}\ep$ for $0\leq j\leq n_k+2$. Finally, take
	\begin{eqnarray*}
		d^{1,f}(\phi^{1,f}_{\alpha(t)}(x,r),\phi^{1,f}_{t-T_k}(x_k,s_k))&=&d^{1,f}\left((f^{N_k}(x),r+\alpha(t)-N_k),(x_k,s_k+t-T_k)\right)
	\end{eqnarray*}
	\begin{eqnarray*}
		&=&d^{1,f}\left((f^{j+N_k}(x),r+\alpha(t)-N_k-j),(f^j(x_k),s_k+t-T_k-j)\right)\\
		&\leq & d^{1,f}\left((f^{j+N_k}(x),r+\alpha(t)-N_k-j),(f^{j+N_k}(x_k),s_k+t-T_k-j)\right)\\
		& &+d^{1,f}\left((f^{j+N_k}(x_k),s_k+t-T_k-j),(f^{j}(x_k),s_k+t-T_k-j)\right)\\
		&\leq & \vert r+\alpha(t)-N_k-j-(s_k+t-T_k-j)\vert+(s_k+t-T_k-j)d(f^{j+N_k+1}(x),f^{j+1}(x_k))\\
		& & +(1-s_k-t+T_k+j)d(f^{j+N_k}(x),f^j(x_k))\\
		&< & \delta^{\prime}+\tfrac{1}{2}(1-(s_k+t-T_k-j))\ep+\tfrac{1}{2}(s_k+t-T_k-j)\ep\leq \tfrac{1}{2}\ep+\tfrac{1}{2}\ep=\ep.
	\end{eqnarray*}
	Hence $(x,r)\in Sh(\phi^{1,f})$.
\end{proof}

\begin{proof}[Proof of Theorem \ref{teo2.5}]

Given $\ep>0$ we denote by $Sh(\phi\,,\ep)$ the set of points $p\in X$ such that the flow has the {\em POTP} through a subset $\{p\}$ (see Definition \ref{3.6}). Note that 
\begin{equation}\label{5}
Sh(\phi)=\bigcap_{\ep>0}Sh(\phi,\,\ep).
\end{equation}
Let $\ep_0>0$. We can suppose $Sh(\phi)\neq \emptyset$. Given $x\in Sh(\phi)$, since $x\in Sh(\phi,\,\frac{\ep_0}{2})$, by Lemma \ref{lemma3.7} there is $\delta_{x,\,\ep_0} > 0$ such that every $(\delta_{x,\,\ep_0},\,1)$-pseudo-orbit passing through $B[x,\,\delta_{x,\,\ep_0} ]$ can be $\ep_0$-shadowed. It follows that $B(x,\,\delta_{x,\,\ep_0} )\subset Sh(\phi,\,\ep_0)$. So, for every $\ep_0>0$
$$Sh(\phi,\,\ep_0)=A(\ep_0)\cup B(\ep_0),$$
where $A(\ep_0)=\displaystyle{\bigcup_{x\in Sh(\phi)}}B(x,\,\delta_{x,\,\ep_0} )$ and $B(\ep_0)=Sh(\phi,\,\,\ep_0)\setminus A(\ep_0)$. Moreover, note that $A(\ep_0)$ is open and $B(\ep_0)\subset Sh(\phi,\,\ep_0)\setminus Sh(\phi)$. By (\ref{5}) we have $\displaystyle{\bigcap_{\ep_0>0}}B(\ep_0)=\emptyset$ and 

$$Sh(\phi)=\bigcap_{n\in\mathbb{N}}Sh(\phi,\,\tfrac{1}{n})=\bigcap_{n\in\mathbb{N}}A(\tfrac{1}{n})\cup B(\tfrac{1}{n})=\bigcap_{n\in\mathbb{N}}A(\tfrac{1}{n}).$$
That is, $Sh(\phi)$ is a $G_{\delta}$ set of $X$.
\end{proof}

\begin{proof}[Proof of Corollary \ref{coroG}]
By Theorem \ref{teo2.5}, if $\phi^{1,f}$ is the suspension of $f$ under the constant function 1, then $Sh(\phi^{1,f})$ is a $G_{\delta}$ set of $X^{1,f}$. Thus there is a sequence $(A_n)_{n\in \mathbb{N}}$ of open sets in $X^{1,f}$ with the following property $Sh(\phi^{1,f})=\bigcap_{n=1}^{\infty}A_n$. So, by Theorem \ref{suspe1} we have $$p(Sh(f)\times [0,1])=\bigcap_{n=1}^{\infty}A_n,$$
where $p:X\times [0,1]\rightarrow X^{1,f}$ is the quotient map of $X^{1,f}$. Moreover, since $Sh(f)$ is invariant with respect to $f$ we obtain that $Sh(f)\times [0,1]=p^{-1}(p(Sh(f)\times [0,1]))$. Hence \begin{equation}\label{Gdelta}
Sh(f)\times [0,1]=\bigcap_{n=1}^{\infty}p^{-1}(A_n).
\end{equation}
Next, given $z\in Sh(f)$, by ($\ref{Gdelta}$) we have that $z\times [0,1]\subset p^{-1}(A_n)$ for every $n\in\mathbb{N}$. Since $p^{-1}(A_n)$ are open sets, there exists $\varepsilon_{z,n}\leq \frac{1}{n}$ such that $B(z,\varepsilon_{z,n})\times [0,1]\subset p^{-1}(A_n)$. Finally, if $V_n=\bigcup_{z\in Sh(f)}B(z,\varepsilon_{z,n})$ it follows that
$$Sh(f)=\bigcap_{n=1}^{\infty}V_n.$$
This concludes the proof.
\end{proof}

\begin{proof}[Proof of Theorem \ref{teo2.9}]
Suppose that $Sh^+(\phi)\neq\emptyset$. Let $p\in Sh^+(\phi)$ and $q\in X$. Let $\ep > 0$ and take $\delta > 0$ from the forward shadowableness of $p$. Let $(x_i,\, t_i)_{i=0}^{\infty}$ a  forward $(\delta,\,1)$-pseudo-orbit passing through $q$ and let $(y_i,\,s_i)_{i=0}^m$ a $(\delta,\,1)$-chain such that $y_0 = p$ and $y_m= q$. We have that the sequence of pairs $(z_j,\,r_j)_{j=0}^{\infty}$ given by 

\[
  (z_j,\,r_j) =
 \begin{cases} (y_j,\,s_j)\quad \quad \ \mbox{ if } 0\leq j< m,\\
               (x_{j - m},\, t_{j - m})\, \mbox{ if }  j\geq m.
  \end{cases}
\] 
is a forward $(\delta,\,1)$-pseudo-orbit passing through $p$. We set 
 
   \[\hat{r}_i =
      \begin{cases}
        \displaystyle{\sum_{j=0}^{i-1}} r_j & i > 0,\\
		    0				             & i = 0,
      \end{cases}
    \]
and
    
    \[\hat{t}_i =
      \begin{cases}
        \displaystyle{\sum_{j=0}^{i-1}} t_j & i > 0,\\
		    0				             & i = 0.
      \end{cases}
    \]
Hence there are $h\in \mathrm{Rep}$ and a point $y\in X$ such that $d(p\star t,\,\phi_{h(t)}(y))\leq\ep$, for $t\in[0,\,\infty)$. Let $g(t) = h(t + \hat{r}_m) - h(\hat{r}_m)$.  

Note that 
$$\hat{r}_{m + k} - \hat{r}_m = \sum_{j = 0}^{m + k -1}r_j-\sum_{j=0}^{k-1}r_j = \sum_{j=m}^{m+k-1}t_j = \hat{t}_k.$$
So, if $\hat{t_k}\leq t < \hat{t}_{k+1}$ with $k\geq 0$, then $\hat{r}_{m+k}\leq t + \hat{r}_m < \hat{r}_{k + m +1}$ and therefore

\begin{align*}
d(q\star t,\, \phi_{g(t)}(\phi_{h(\hat{r}_m)}(y)))&=d(\phi_{t-\hat{t}_k}(x_k),\,\phi_{h(t+\hat{r}_m) - h(\hat{r}_m)}(\phi_{h(\hat{r}_m)}(y)))\\
&= d(\phi_{t + \hat{r}_{m}-\hat{r}_{m+k}}(z_{m+k}),\,\phi_{h(t+\hat{r}_m)}(y))\\
&\leq\ep.
\end{align*}
We have shown that the given forward $(\delta,\,1)$-pseudo-orbit passing through $q$ can be $\ep$-shadowed by the point $\phi_{h(\hat{t}_m)}(y)$ and as this $\ep$ was arbitrary we conclude that $q$ forward shadowable. That is $q\in Sh^+(\phi)$. This completes the proof.
\end{proof}

\begin{proof}[Proof of Theorem \ref{teo2.3.6}] 
As $X$ is a Baire Space, it is enough to proof that for any pairs of open sets $U$ and $V$ of $X$, there exists a non-negative $T$ with $\phi_T(U)\cap V\neq\emptyset$. By hypothesis there exists at least one shadowable point $x$. Choose  two points $p\in U$ and $q\in V$ and let $\epsilon>0$ such that $B(p,\,\ep)\subseteq U$ and $B(q,\,\ep)\subseteq V$. Let $\delta> 0$ from the shadowableness of $x$ with respect to $\ep$. By chain transitivity there exists a $(\delta,\,1)$-chain $(x_i,\,u_i)_{i=0}^m$ from $p$ to $x$ and a $(\delta,\,1)$-chain $(y_j,\,r_j)_{j=0}^n$ from $x$ to $q$. Define a $(\delta,\,1)$-pseudo-orbit $(z_i,\,t_i)_{i\in\mathbb{Z}}$ passing through $x$ as follows:

\[
(z_k,\,t_k) = \begin{cases} 
(\phi_{k+m}(x_0),\,1) &\mbox{if } k< -m,\\
(x_{m+k},\,u_{m+k}) & \mbox{if } -m\leq k< 0,\\
(y_k,\,r_k) &\mbox{if } 0\leq k\leq n  -1,\\
  (\phi_{k-n}(y_n),\,1)&\mbox{if } k\geq n. 
\end{cases}
\]
Then there are $y\in X$ and $h\in\mathrm{Rep}$ such that 
\[d(\phi_{t-s_k}(x),\,\phi_{h(t)}(y))\leq \ep,\ \forall\,s_k\leq t<s_{k+1},\]
where $(s_k)_{k\in \mathbb{Z}}$ is the sequence of sums associated to $(t_k)_{k\in \mathbb{Z}}$. In particular, we have 
\[
d(p,\,\phi_{h(s_{-m})}(y))\leq\ep
\]
and
\[
d(q,\,\phi_{h(s_{n})}(y))\leq\ep.
\] 
Set $T = -h(s_{-m}) +h(s_n)$ which is nonnegative since $h\in\mathrm{Rep}$. Then $\phi_{h(s_n)}(y)\in \phi_T(U)\cap V$. This concludes the proof.
\end{proof}

%%%%%%%%%%%%%%%%%%%%%%%%%%%%%%%%%%%%%%%%%%%%%%%%%%%%%%%%%%%%%%%%%%%%%%%%%%%

\end{document}